\documentclass[11pt]{amsart}
\usepackage{amsmath,amsthm,amscd,amsfonts,amssymb,enumerate,tikz-cd}
\usepackage{hyperref,cleveref}
\usepackage{xcolor}
\usepackage[labelformat=empty]{caption,subcaption}

\newtheorem{theorem}{Theorem}[section]
\newtheorem{lemma}[theorem]{Lemma}
\newtheorem{proposition}[theorem]{Proposition}
\newtheorem{corollary}[theorem]{Corollary}
\theoremstyle{definition}
\newtheorem{definition}[theorem]{Definition}

\theoremstyle{remark}
\newtheorem{remark}[theorem]{Remark}

\numberwithin{equation}{section}

\allowdisplaybreaks

	\title[ MacLane-Vaqui\'e chains  and Valuation-Transcendental Extensions]{MacLane-Vaqui\'e chains  and Valuation-Transcendental Extensions}
	\author[Sneha Mavi]{Sneha Mavi}
	\address{Department of Mathematics\\ University of Delhi\\  Delhi-110007, India.}
	\email{mavisneha@gmail.com}
	
	\author[Anuj Bishnoi]{Anuj Bishnoi$^\ast$}
	\address{Department of Mathematics\\  University of Delhi \\   Delhi-110007, India.}
	\email{abishnoi@maths.du.ac.in}
		\thanks{$^\ast$Corresponding author, E-mail address: abishnoi@maths.du.ac.in}
		
		\date{09,19,2022}
		

\begin{document}
	\begin{abstract}
	In this paper, 	for a  valued field $(K,v)$ of arbitrary rank  and an extension $w$ of $v$ to $K(X),$  we  give a connection  between complete sets of ABKPs for   $w$   and MacLane-Vaqui\'e chains  of $w.$ 
	\end{abstract}

	\maketitle
	
	\section{Introduction }    
	  In  1936,   MacLane \cite{M}, proved that an extension $w$ of a discrete rank one valuation  $v$ of $K$ to $K[X]$ can be obtained as a chain of augmentations
	 \begin{align*}
	 	w_0\xrightarrow{\phi_1,\gamma_1} w_1\xrightarrow{\phi_2,\gamma_2}\cdots \longrightarrow w_{n-1}\xrightarrow{\phi_n,\gamma_n} w_n\longrightarrow\cdots
	 \end{align*}
for some suitable  key polynomials $\phi_i\in K[X]$ with respect to the  intermediate valuations, and  values $\gamma_{i}$ belonging to  some  totally ordered abelian group containing the value group of $v$ as an ordered subgroup. Later,   Vaqui\'e generalized MacLane's theory to arbitrary valued fields (see \cite{V}). 
	Recently,  Nart gave a survey of generalized MacLane-Vaqui\'e  theory   in \cite{EN} and \cite{EN1}. 
	Starting with a valuation $w_0$ which admit key polynomials of degree one, Nart also introduced {\bf MacLane-Vaqui\'e  chains}, consisting of a mixture of ordinary and limit augmentations  (see Definition \ref{1.1.13}). The main result,  Theorem 4.3  of \cite{EN1} says that, all extensions  $w$  of $v$ to $K[X]$ fall exactly in one of the following categories:
	\begin{enumerate}[(i)]
		\item It is the last valuation of a complete finite MacLane-Vaqui\'e chain, i.e., after a finite number $r$ of augmentation steps, we get $w_r=w.$
		\item After a finite number $r$ of augmentation steps, it is the stable limit of a continuous family of augmentations  of $w_r$ defined by key polynomials of constant degree.
		\item  It is the stable limit of a complete infinite MacLane-Vaqui\'e chain.
	\end{enumerate}
	    In this paper, we  study  MacLane-Vaqui\'e chains of the first type and prove that a precise complete finite  MacLane-Vaqui\'e chain can be obtained   using a given complete set $\{Q_i\}_{i\in\Delta}$ of ABKPs for $w$ such that $\Delta$ has a last element. Conversely, if $w$ is the last valuation of some complete finite   MacLane-Vaqui\'e chain, then it is possible to derive, from the data of this chain, a complete set $\{Q_i\}_{i\in\Delta}$ of ABKPs for $w$ such that $\Delta$ has a last element.
	
	To state the main results of the paper,  we first recall some notation, definitions and  preliminary results.

	\section{Background and notation}
	
	Throughout  $(K,v)$ denote a    valued field of arbitrary rank  with value group $\Gamma_ v,$   residue field  $k_{v}$ and by $\bar{v}$ we denote  an extension of $v$ to a fixed algebraic closure $\overline{K}$ of $K.$ 
	
	\medskip
An  extension $w$ of $v$ to the simple transcendental extension $K(X)$ of $K$    is said to be {\bf value-transcendental}  if the quotient $\Gamma_{w}/\Gamma_v$ is a torsion-free group and the corresponding residue field extension $k_w|k_v$ is algebraic.    The extension $w$ of $v$ to $K(X)$ is called {\bf residually transcendental} (abbreviated as r.\ t.)  if  $\Gamma_{w}/\Gamma_v$ is a torsion group and
		 $k_{w}| k_v$ is transcendental.
		We call $w$  {\bf valuation-transcendental} if it is either value-transcendental or is residually transcendental.

	 We fix a common extension $\overline{w}$ of $w$ and $\bar{v}$ to $\overline{K}(X).$
 By \cite[Lemma 3.3]{FV-K},  $\overline{w}$  is valuation-transcendental if and only if $w$ is valuation-transcendental.

	\begin{definition}
		For any valuation $w'$ on $K(X)$ taking values in a subgroup of $\Gamma_{w},$ we say that $w'\leq w$ if and only if 
		$$w'(f)\leq w(f),~\forall f\in K[X].$$
		If $w'<w,$ we denote  by $\Phi(w',w)$ the set of all monic polynomials $g\in K[X]$ of minimal degree such that $w'(g)<w(g).$
	\end{definition}
	\begin{definition}\label{1.1.2}
		A pair $(\alpha,\delta)$ in $\overline{K}\times \Gamma_{\overline{w}}$ is called a {\bf  pair of definition} for $w$  if 
	 $\overline{w}=\overline{w}_{\alpha,\delta},$  where $\overline{w}_{\alpha,\delta}$ is a valuation on $\overline{K}[X]$ defined by 
$$\overline{w}_{\alpha,\delta}\left(\sum_{i\geq 0} c_i (X-\alpha)^i\right):=\min_{i\geq 0}\{\bar{v}(c_i)+i\delta\}, \, c_i\in\overline{K}.$$
	\end{definition}

\begin{definition}
For any polynomial $f$ in $K[X],$ we denote
$$\delta(f)=\max\{\overline{w}(X-\alpha)\mid f(\alpha)=0\}.$$ 
This value $\delta(f)$ is independent of the choice of $\overline{w}$  \cite[Proposition 3.1]{JN}.
\end{definition}

	\begin{definition}[\bf Abstract key polynomials]
		A monic polynomial $Q$ in $K[X]$ is said to be an {\bf abstract key polynomial}  (abbreviated as ABKP) for $w$ if for each polynomial $f$ in $K[X]$   with $\deg f< \deg Q$ we have $\delta(f)<\delta(Q).$
	\end{definition}
	It is immediate from the definition that all monic  linear polynomials are ABKPs for $w.$  Also an ABKP for $w$ is an irreducible polynomial (see   \cite[Proposition 2.4]{NS}).
	\begin{definition}
		For a polynomial $Q$ in $K[X]$ the {\bf $Q$-truncation} of $w$ is a map $w_Q:K[X]\longrightarrow \Gamma_w$ defined by 
		$$ w_Q(f):= \min_{0\leq i\leq n}\{w(f_iQ^i)\},$$
		where $f=\sum_{i=0}^{n} f_i Q^i,$ $\deg f_i <\deg Q,$  is the $Q $-expansion of $f.$ 
	\end{definition}
	The $Q$-truncation  $w_Q$  of $w$ need not be a valuation \cite[Example 2.5]{NS}. However,  if $Q$ is an ABKP for $w,$ then $w_Q$ is a valuation on $K(X)$ (see \cite[Proposition 2.6]{NS}). Also any ABKP, $Q$ for $w,$ is also an ABKP for the truncation valuation $w_Q.$

\begin{lemma}[Lemma 2.11, \cite{NS}] \label{1.2.6}
	If $Q$ is an ABKP for $w,$ then every element $F\in\Phi(w_Q,w)$ is also an ABKP for $w$ and $\delta(Q)<\delta(F).$
\end{lemma}

	\begin{definition}
		A family $\Lambda=\{Q_i\}_{i\in\Delta}$ of ABKPs for $w$ is said to be a {\bf complete set of ABKPs} for $w$ if the following conditions are satisfied:
		\begin{enumerate}[(i)]
			\item $\delta(Q_i)\neq \delta(Q_j)$ for every $i\neq j\in\Delta.$
			\item $\Lambda$ is well-ordered with respect to the ordering given by $Q_i< Q_j$ if and only if  $\delta(Q_i)<\delta(Q_j)$ for every $i<j\in \Delta.$ 
			\item For any $f\in K[X],$ there exists some $Q_i \in \Lambda$ such that $\deg Q_i\leq \deg f$ and  $w_{Q_i}(f)=w(f).$
		\end{enumerate}
\end{definition}

\begin{theorem}[Theorem 1.1, \cite {NS}]
	Every valuation $w$ on $K(X)$ admits a complete set of ABKPs for $w.$  
\end{theorem}
 
	\begin{remark} \label{1.1.3}
		As shown in  \cite[Remark 4.6]{MMS} and  \cite[proof of Theorem 1.1]{NS}, there is a complete set  $\Lambda=\{Q_i\}_{i\in\Delta}$ of ABKPs  for  $w$ having the following properties.
		\begin{enumerate}[(i)]
		\item $\Delta=\bigcup_{j\in I}\Delta_j$ with $I=\{0,1,\ldots, N\}$ or $\mathbb{N}\cup\{0\},$ and for each $j\in I$ we have $\Delta_j=\{j\}\cup\vartheta_{j},$ where $\vartheta_j$ is an ordered set without last element or is empty.
		\item $Q_0$ is a monic polynomial of degree one.
		\item For all $j\in I\setminus \{0\}$ we have $j-1<i<j,$ for all $i \in\vartheta_{j-1}.$
		\item All  polynomials $Q_i$ with $i\in\Delta_j$ have the same degree and  have degree strictly  less than  the degree of the polynomials $Q_{i'}$ for every $i'\in\Delta_{j+1}.$
		\item For each $i<i'\in\Delta$ we have $w(Q_i)<w(Q_{i'})$ and $\delta(Q_i)<\delta(Q_{i'}).$
		\end{enumerate}
	\end{remark}
{\bf All complete set of ABKPs used in this paper will be assumed to satisfy these properties.}

	Even though the set $\{Q_i\}_{i\in\Delta}$  of ABKPs  for $w$  is not unique, the cardinality of $I$ and the degree of an abstract key polynomial $Q_i$ for each $i\in I$ are uniquely determined by $w.$
	
	The ordered set $\Delta$ has a last element if and only if  the set $I=\{0,1,\ldots,N\}$ is finite and  
		$\Delta_N=\{N\},$ i.e., $\vartheta_N=\emptyset.$

	\begin{definition}[\bf Limit ABKPs]\label{1.1.22}
	Let $\Lambda=\{Q_i\}_{i\in\Delta}$ be a complete set of ABKPs for $w.$	For	an element $i\in I\setminus\{0\},$ we say that  $Q_i$ is a {\bf limit ABKP} for $w$ if $i$ has no immediate predecessor in $\Delta;$ that is,
			 $\vartheta_{i-1}\neq \emptyset.$
	\end{definition}

\medskip
We now recall  the  notion of key polynomials which was first introduced by MacLane  in 1936 and   generalized by Vaqui\'e in 2007 (see \cite{M} and \cite{V}).  
\begin{definition}[\bf Key polynomials]
	For a valuation $w$ on $K(X)$ and polynomials $f,$ $g$ in $K[X],$ we say that
	\begin{enumerate}[(i)]
		\item  $f$ and $g$ are $w$-equivalent and write $f\thicksim_{w} g$ if $w(f-g)>w(f)=w(g).$
		\item $g$ is $w$-divisible by $f$  (denoted by $f\mid_{w}g$) if there exists some polynomial $h \in K[X]$ such that $g\thicksim_{w} fh.$
		\item  $f$ is $w$-irreducible if for any $h,\, q\in K[X],$ whenever $f\mid_{w} hq,$ then either $f\mid_{w}h $ or $f\mid_{w}q.$
		\item $f$ is $w$-minimal if for every nonzero polynomial $h\in K[X],$ whenever $f\mid_{w}h,$ then $\deg h\geq \deg f.$
	\end{enumerate}
Any monic polynomial $f$  satisfying (iii) and (iv) is called a (MacLane-Vaqui\'e)  {\bf key polynomial} for $w.$
	\end{definition}
In view of  Proposition 2.10 of  \cite{Ma},  any ABKP, $Q$ for $w$ is a key polynomial for $w_Q$ of minimal degree.
Let $\operatorname{KP}(w)$ denote the set of all key polynomials for $w.$ For any $\phi\in \operatorname{KP}(w),$ we denote by $[\phi]_w$ the set of all key polynomials which are $w$-equivalent to $\phi.$

The existence of key polynomials is characterized as follows.
\begin{theorem}[Theorem 4.4, \cite{EN}]\label{1.1.1}
	A valuation $w$ on $K(X)$ has $\operatorname{KP}(w)\neq\emptyset$ if and only if  it is valuation-transcendental.
\end{theorem}
\begin{definition}[\bf Ordinary augmentation]\label{1.1.15}
	Let $\phi$ be a key polynomial for a valuation $w'$ on $K(X)$ and let $\gamma> w'(\phi)$ be an element of a totally ordered abelian group $\Gamma$ containing $\Gamma_{w'}$ as an ordered subgroup. The map $w: K[X]\longrightarrow \Gamma\cup\{\infty\}$ defined by 
	$$w(f):=\min\{w'(f_i)+i\gamma\},$$ where
	 $f=\sum_{i\geq 0}f_i \phi^i, $ $\deg f_i<\deg \phi,$ is the $\phi$-expansion of $f\in K[X],$  gives a valuation on $K(X)$ (see \cite[Theorem 4.2]{M}) called the   {\bf ordinary augmentation of $w$}  and is  denoted  by $w=[w'; \phi,\gamma].$
\end{definition}
 Clearly,  $\phi\in\Phi(w',w).$  On the other hand,
the polynomial $\phi$ is a key polynomial of minimal degree for the augmented valuation $w$ (see \cite[Corollary 7.3]{EN}). 

 If $\phi$ is a key polynomial for $w$ of minimal degree, then we define
 $$\deg(w):=\deg \phi.$$ 
 
 \begin{theorem}[Theorem 1.15, \cite{V}]\label{1.1.19}
 	Let $w$ be a valuation on $K(X)$ and $w'<w.$ Then 
 	any $\phi\in \Phi(w',w)$ is a key polynomial for $w'$ and  $$w'<[w';\phi,w(\phi)]\leq w.$$  For any non-zero polynomial $f\in K[X],$ the equality 
 	$w'(f)=w(f)$ holds if and only if $\phi\nmid_{w'} f.$
 \end{theorem}
 \begin{corollary}[Corollary 2.5, \cite{EN1}]\label{1.1.16}
 	Let $w'<w$ be as above. Then
 	\begin{enumerate}[(i)]
 		\item  $\Phi(w',w )=[\phi]_{w'}$ for all $\phi\in\Phi(w',w ).$
 		\item If $w'<\mu\leq w$ is a chain  of valuations, then $ \Phi(w',w )=\Phi(w',\mu).$
 		In particular,  for all $f\in K[X]$ we have 
 		\begin{align*}
 			w'(f)=w(f) \iff w'(f)=\mu(f). 
 		\end{align*}
 	\end{enumerate}
 \end{corollary}
 
\bigskip

Let $(K,v)$ be a valued field and $w'$ be an extension of $v$ to $K(X).$ We now recall the definition of a  continuous family of augmentations of $w'$  (\cite{V, EN}).
\begin{definition}\label{1.1.14}
	Let $w'$ be a valuation-transcendental valuation on $K(X).$ A  {\bf continuous family of augmentations} of $w'$ is a family of ordinary augmentations of $w'$  $$\mathcal{W} =(\rho_i=[w';\chi_i,\gamma_i])_{i\in	\mathbf{A}},$$ indexed by a totally ordered set $\mathbf{A},$ satisfying the following conditions:
	\begin{enumerate}[(i)]
		\item The set $\mathbf{A}$  has no last element.
		\item All  key polynomials $\chi_i\in \operatorname{KP}(w')$ have the same degree.
		\item For all $i<j$ in $\mathbf{A},$ $\chi_j$ is a  key polynomial for $\rho_i$ and satisfies:
		$$\chi_j\not\thicksim_{\rho_i}\chi_i~\text{and}~ \rho_j=[\rho_i;\chi_j,\gamma_j].$$
	\end{enumerate}
\end{definition}
	The common degree $m=\deg\chi_i,$ for all $i,$ is called the {\bf stable degree} of the family $\mathcal{W}$ and is denoted by $\deg (\mathcal{W}).$
	
A polynomial $f\in K[X]$ is said to be $\mathcal{W}$-{\bf stable} with respect to the family $\mathcal{W}=(\rho_i)_{i\in\mathbf{A}}$  if 
$$\rho_i(f)=\rho_{i_0}(f),~ \text{for every}~ i\geq i_0$$ for some index $i_0\in\mathbf{A}.$ This stable value is denoted by $\rho_\mathcal{W}(f).$ By Corollary \ref{1.1.16} (ii), a polynomial $f\in K[X]$  is {\bf $\mathcal{W}$-unstable} if and only if  
$$\rho_i(f)<\rho_j(f),\hspace{5pt} \forall i<j\in\mathbf{A}.$$  The minimal degree of an $\mathcal{W}$-unstable polynomial is denoted by $m_{\infty}.$ If  all polynomials are $\mathcal{W}$-stable, then we set $m_{\infty}=\infty$. 
\begin{remark}\label{1.1.18}
The following properties hold for any  continuous family $\mathcal{W} =(\rho_i)_{i\in\mathbf{A}}$ of augmentations (see  \cite[p.\ 9]{EN1}):
\begin{enumerate}[(i)]

	\item For all $i\in\mathbf{A},$ $\chi_i$ is a key polynomial for $\rho_i$ of minimal degree.
		\item $\Phi(\rho_i,\rho_j)=[\chi_j]_{\rho_i}$ $\forall$ $i<j\in\mathbf{A}.$
	\item All the value groups $\Gamma_{\rho_i}$ coincide and   the common value group is denoted by $\Gamma_{\mathcal{W}}.$ 
\end{enumerate}
\end{remark}
\begin{definition}[\bf MacLane-Vaqui\'e limit key polynomials]
	Let $\mathcal{W}$ be a continuous family of augmentations of a valuation $w'.$ 
 A   monic $\mathcal{W}$-unstable polynomial of minimal degree is called a {\bf MacLane-Vaqui\'e  limit key polynomial} (abbreviated as MLV) for $\mathcal{W}.$
\end{definition}
  We denote by $\operatorname{KP}_{\infty}(\mathcal{W})$  the set of all MLV limit key polynomials. Since the product of stable polynomials is stable, so all MLV limit key polynomials are irreducible in $K[X].$ 
 \begin{definition}
 	We say that $\mathcal{W}$ is an {\bf essential continuous family} of augmentations if $\deg (\mathcal{W})=m<m_{\infty}<\infty.$
 \end{definition}

Let $\mathcal{W}$ be an essential continuous family of augmentations of a valuation $w'.$ Then 	$\mathcal{W}$ admit MLV limit key polynomials.
If $Q$ is any MLV limit key polynomial, then any polynomial $f$ in $K[X]$ with $\deg f<\deg Q$ is  $\mathcal{W}$-stable.
\begin{definition}[\bf Limit augmentation]
	Let   $Q$ be any MLV limit key polynomial for an essential continuous family of augmentations $\mathcal{W} =(\rho_i)_{i\in\mathbf{A}}$ of $w'$  and   $\gamma>\rho_i(Q),$ for all $i\in\mathbf{A},$ be an element of  a totally ordered abelian group $\Gamma$  containing $\Gamma_{\mathcal{W}}$ as an  ordered subgroup. Then, the map $w: K[X]\longrightarrow\Gamma\cup\{\infty\}$ defined by
	$$w(f)=\min_{i\geq 0}\{\rho_\mathcal{W}(f_i)+i\gamma\},$$
where  $f=\sum_{i\geq 0} f_iQ^i,$ $\deg f_i<\deg Q,$ is the $Q$-expansion of $f\in K[X],$ gives a valuation on $K(X),$  called  the {\bf limit augmentation} of $\mathcal{W}$ and we denote it by $w=[\mathcal{W}; Q, \gamma].$ 

\end{definition}
 Note that $w(Q)=\gamma$ and $\rho_i<w$ for all $i\in\mathbf{A}.$ On the other hand,  $Q$ is a key polynomial for $w$ of minimal degree \cite[Corollary 7.13]{EN}.
\vspace{.20pt}

We now recall the definition of MacLane-Vaqui\'e chains given by  Nart in \cite{EN1}. For this,  we first
consider a finite, or countably infinite, chain of mixed augmentations
\begin{align}\label{1.12}
	w_0\xrightarrow{\phi_1,\gamma_1} w_1\xrightarrow{\phi_2,\gamma_2}\cdots \longrightarrow w_{n}\xrightarrow{\phi_{n+1},\gamma_{n+1}} w_{n+1}\longrightarrow\cdots
\end{align}
in which every valuation is an augmentation of the previous one and is of one of the following type:
\begin{itemize}
	\item  Ordinary augmentation: $w_{n+1}=[w_n; \phi_{n+1},\gamma_{n+1}],$ for some $\phi_{n+1}\in \operatorname{KP}(w_n).$ 
	\item Limit augmentation:  $w_{n+1}=[\mathcal{W}_n; \phi_{n+1},\gamma_{n+1}],$ for some $\phi_{n+1}\in \operatorname{KP}_{\infty}(\mathcal{W}_n),$ where $\mathcal{W}_n=(\rho_i)_{i\in\mathbf{A}_n}$ is an essential continuous family of augmentations of $w_n.$ 
\end{itemize}
Let $\phi_0\in \operatorname{KP}(w_0)$ be a key polynomial of minimal degree and let $\gamma_0=w_0(\phi_0).$
 Then, in view of  Theorem \ref{1.1.19},  Proposition 6.3 of \cite{EN}, Proposition 2.1, 3.5 of \cite{EN1} and Corollary \ref{1.1.16},  we have the   following properties of a  chain (\ref{1.12}) of augmentations.
 \begin{remark}\label{1.1.20}
\begin{enumerate}[(i)]
	\item  $\gamma_n=w_n(\phi_n)<\gamma_{n+1}.$
	\item For all $n\geq 0,$ the polynomial $\phi_n$ is a key polynomial for $w_n$ of minimal degree and therefore
	$$\deg(w_n)(=\deg \phi_n)~\text{divides}~\deg (\Phi(w_n,w_{n+1})).$$
	\item \begin{equation*}
		\Phi(w_n,w_{n+1})=
		\begin{cases}
			[\phi_{n+1}]_{w_n},~ \text{if}~ w_n\rightarrow w_{n+1}~ \text{is an  ordinary augmentation}\\
		\displaystyle\Phi(w_n,\rho_i), ~\text{for all $i\in\mathbf{A}_n,$}~\text{ if}~ w_n\rightarrow w_{n+1} ~\text{is a limit augmentation}.
		\end{cases}
	\end{equation*}
	\item \begin{equation*}
			\deg(\Phi(w_n,w_{n+1}))=
			\begin{cases}
				\deg\phi_{n+1},~\text{if}~ w_n\rightarrow w_{n+1}~ \text{is an ordinary augmentation}\\
				\displaystyle\deg(\mathcal{W}_n),~\text{ if}~ w_n\rightarrow w_{n+1} ~\text{is a limit augmentation}.
			\end{cases}
	\end{equation*}
\end{enumerate}
\end{remark}

\begin{definition}[\bf MacLane-Vaqui\'e chains]\label{1.1.13}
	A finite, or countably infinite chain of mixed augmentations as in (\ref{1.12}) is called a {\bf MacLane-Vaqui\'e  chain} (abbreviated as MLV), if every augmentation step satisfies:
	\begin{enumerate}[(i)]
		\item if $ w_n\rightarrow w_{n+1}$ is an ordinary augmentation, then $\deg (w_n)<\deg(\Phi(w_n,w_{n+1})).$
		\item if $ w_n\rightarrow w_{n+1}$ is a limit augmentation, then $\deg (w_n)=\deg(\Phi(w_n,w_{n+1}))$ and $\phi_n\notin\Phi(w_n,w_{n+1}).$
	\end{enumerate}
	An MLV  chain is said to be {\bf complete} if  $w_0$ is a   depth zero valuation; that is, $w_0=\overline{w}_{a,\delta}$ for some $a\in K$ and $\delta\in\Gamma$ (see Definition \ref{1.1.2}).
\end{definition}
\section{Statements of Main Results}
In the following result, we use a  complete set of ABKPs for a valuation $w$ to    construct  a complete finite MLV chain whose last valuation is $w.$
\begin{theorem}\label{1.1.11}
	Let $(K,v)$ be a valued field and let $w$ be an extension of $v$ to $K(X).$ If $\{Q_i\}_{i\in\Delta}$ is a complete set of ABKPs for $w$  such that $N$ is the last element of $\Delta,$ then 
		\begin{align}\label{1.13}
		w_0\xrightarrow{Q_1,\gamma_1} w_1\xrightarrow{Q_2,\gamma_2}\cdots \longrightarrow w_{N-1}\xrightarrow{Q_N,\gamma_N} w_N=w,
	\end{align}
	 is a complete finite MLV chain of $w$ such that 
	\begin{enumerate}[(i)]
	 	\item  if $\vartheta_{j}=\emptyset,$ then $w_j\longrightarrow w_{j+1}$ is an ordinary augmentation. 
	 	\item if $\vartheta_{j}\neq \emptyset,$ then $w_j\longrightarrow w_{j+1}$ is a limit augmentation. 
	\end{enumerate}
In both cases, $w_{j+1}=w_{Q_{j+1}}$ and $\gamma_{j+1}=w(Q_{j+1}).$
\end{theorem}
The converse of the above result also holds.
\begin{theorem}\label{1.1.12}
	Let $(K,v)$ be a valued field and let $w$ be an extension of $v$ to $K(X).$ If 
	\begin{align}\label{1.14}
		w_0\xrightarrow{\phi_1,\gamma_1} w_1\xrightarrow{\phi_2,\gamma_2}\cdots \longrightarrow w_{N-1}\xrightarrow{\phi_N,\gamma_N} w_N=w,
	\end{align}
is a complete finite MLV chain of $w,$ 
	 then $\Lambda=\{Q_i\}_{i\in\Delta}$ forms a complete set of ABKPs for $w,$ where
	 \begin{enumerate}[(i)]
	 	\item $\Delta=\bigcup_{j\in I}\Delta_j,$ with $I=\{0,1,\ldots, N\},$ and  $\Delta_j=\{j\}\cup\mathbf{A}_j$ for all $j\in I.$ Moreover,  $Q_j=\phi_j$ for all $j\in I.$
	 	\item if $w_j\longrightarrow w_{j+1}$ is an ordinary augmentation, then  $\vartheta_{j}=\emptyset.$
	 	\item if $w_j\longrightarrow w_{j+1}$ is a limit augmentation, with respect to an essential continuous family $\mathcal{W}_j=(\rho_i)_{i\in\mathbf{A}_j}$ of augmentations of $w_j,$ then $\vartheta_{j}=\mathbf{A}_j.$ Moreover, $Q_i=\chi_i$ for all $i\in\mathbf{A}_j,$ where $\chi_i$ is the key polynomial for $w_j$ such that $\rho_i=[w_j;\chi_i,w(\chi_i)].$
	 \end{enumerate} 
\end{theorem}

It is known that if  $\{Q_i\}_{i\in\Delta}$ is a complete set of ABKPs for $w,$ then $w$ is a valuation-transcendental extension of $v$ to $K(X)$ if and only if $\Delta$ has a last element, say, N, and then 	 $w=w_{Q_N}$ (see \cite[Theorem 5.6]{MMS}). Therefore, as an immediate consequence of Theorems \ref{1.1.11} and \ref{1.1.12}, we have the following result.
\begin{corollary}
	Let $(K,v)$ and $(K(X),w)$ be as above. Then the following are equivalent:
	\begin{enumerate}[(i)]
		\item The extension $w$ is valuation-transcendental.
		\item There exist a complete set  $\{Q_i\}_{i\in\Delta}$ of ABKPs for $w$ such that $\Delta$ has a last element.
		\item  The extension $w$ is the last valuation of a complete finite MLV chain.
	\end{enumerate}
\end{corollary}
\begin{definition}
	A valuation $w$ on $K(X)$ is said to be {\bf inductive} if it is the last valuation of a  complete finite MLV chain, all whose
	 augmentations are ordinary.
\end{definition}

The following result is a consequence of   Theorem \ref{1.1.11} and Theorem \ref{1.1.12}.
\begin{corollary}\label{1.1.21}
	Let $(K,v)$ be a valued field and  $w$  an extension of $v$ to $K(X).$
	Then $w$ is inductive if and only if it admits a finite complete set of ABKPs.
\end{corollary}

\bigskip

	\section{Proof of Main Results}
	Let $(K,v)$ be a valued field and  $(\overline{K},\bar{v})$ be as before.  Let $w$ be an extension of $v$ to $K(X)$  and $\overline{w}$ be a common extension of $w$ and $\bar{v}$ to $\overline{K}(X).$ With notations and  definitions as in the previous section,  we first give some preliminary results which will be used to prove the main results. 

In the following result we recall some basic properties of ABKPs for  $w$   (see  Proposition 3.8, Corollary  3.13 and Theorem 6.1 of  \cite{JN1}).
\begin{proposition}\label{2.1.6}
	For  ABKPs, $Q$ and $Q'$ for $w$ the following holds:
	\begin{enumerate}[(i)]
		\item If  $\delta(Q)<\delta(Q'),$ then $w_Q(Q')<w(Q').$ 
		\item If $\deg Q = \deg Q',$ then $$w(Q)<w(Q')\iff w_Q(Q')<w(Q')\iff \delta(Q)<\delta(Q').$$ Hence $Q'\in\Phi(w_Q,w)$ in this case.
		\item  Suppose that  $\delta(Q)<\delta(Q').$ For any polynomial $f\in K[X],$  we have $$w_Q(f)\leq w_{Q'}(f)\leq w(f).$$ Moreover,
		  if $w_{Q'}(f)<w(f),$ then $w_Q(f)<w_{Q'}(f).$	
		\item If $Q'\in \Phi(w_Q,w),$ then $Q$ and $Q'$ are key polynomials for $w_Q.$ Moreover, $w_{Q'}=[w_Q; Q',w(Q')].$
	\end{enumerate}
\end{proposition}
The following result gives a comparison between key polynomials and ABKPs.
\begin{theorem}[Theorem 2.17, \cite{Ma}]\label{2.1.2}
	Suppose that $w'<w$ and $Q$ is a key polynomial for $w'.$ Then $Q$ is an ABKP polynomial for $w$ if and only if it satisfies one of the following two conditions:
	\begin{enumerate}[(i)]
		\item $Q\in \Phi(w',w),$
		\item $Q\notin \Phi(w',w)$ and $\deg Q=\deg w'.$
	\end{enumerate}
	In the first case $w_Q=[w'; Q, w(Q)].$ In the second case $w_Q=w'.$
\end{theorem}

	\begin{proof}[Proof of Theorem \ref{1.1.11}]
		Let $\Lambda=\{Q_i\}_{i\in\Delta}$ be a complete set of ABKPs for $w$ with $N$ the last element of $\Delta.$ If $N=0,$  then $w=w_{Q_0}$ is a depth zero valuation and the result holds trivially.
		
		 Assume now that $N\geq 1.$ Then by Remark \ref{1.1.3} (i), $\Delta=\bigcup_{j=0}^{N}\Delta_j,$ where  $\Delta_j=\{j\}\cup\vartheta_{j}$ and $\vartheta_j$ is either  empty or an ordered set without a last element. Since for each $i\in\Delta,$ $Q_i$ is an ABKP for $w,$ so $w_{Q_i}$ is a valuation on $K(X)$ and we  denote it by $w_i.$ Since $Q_i\notin\Phi(w_i,w),$  Theorem \ref{2.1.2} shows that $Q_i$ is a key polynomial of minimal degree for $w_i.$ Hence, $\deg (w_i)=\deg Q_i$ for all $i\in\Delta.$

			Suppose first that $\vartheta_j=\emptyset$ for some $0\leq j\leq N-1.$
			We first show that $Q_{j+1}\in\Phi(w_j,w).$ Since $\delta(Q_j)<\delta(Q_{j+1}),$  Proposition \ref{2.1.6} (i) shows that
			 $w_{j}(Q_{j+1})<w(Q_{j+1}).$ 
			 Therefore, it is enough to show that for any polynomial
			$f\in K[X]$ with $\deg f<\deg Q_{j+1},$ we have $w_j(f)=w(f).$ For such a polynomial $f,$ by 
			 definition of the complete set $\Lambda$ of ABKPs,  there exists some $Q_i\in\Lambda$  such that $\deg Q_i\leq\deg f$ and
			 \begin{align*}
			 	w_{i}(f)=w(f).
			 \end{align*}
		 Since $\deg Q_i< \deg Q_{j+1},$ we have $i< j+1.$  As $\vartheta_{j}=\emptyset,$ we necessarily have $i\leq j.$ By Proposition \ref{2.1.6} (iii) and the above equality, we have $$w_i(f)\leq w_j(f)\leq w(f)=w_i(f),$$ so that $w_j(f)=w(f).$ This proves that $Q_{j+1}\in \Phi(w_j,w).$ 
		 
		 By Proposition \ref{2.1.6} (iv), $Q_j$ and $Q_{j+1}$ are key polynomials for $w_j$ and \[w_{j+1}=[w_j;Q_{j+1}, w(Q_{j+1})] \] is an ordinary augmentation of $w_j.$ The condition of a an MLV chain is fulfilled because $\deg (w_j)=\deg Q_j<\deg Q_{j+1}=\deg(\Phi(w_j,w_{j+1})).$
		
		Assume now that $\vartheta_j\neq \emptyset$ for some $0\leq j<N.$ By Proposition \ref{2.1.6} (i), for each $i<i'$ in $\vartheta_{j}$ the ABKPs $Q_i,$ $Q_{i'}$ satisfies \[w_j(Q_i)<w(Q_i)=w_i(Q_i),  \, \, \, \, \, ~~ w_i(Q_{i'})<w(Q_{i'})=w_{i'}(Q_{i'}).\] 
		By  a completely analogous argument as above, we deduce that $Q_i\in\Phi(w_j,w),$ $Q_{i'}\in\Phi(w_i,w)$ and \[w_i=[w_j;Q_i,w(Q_i)], \, \, \, \, \, ~ ~ w_{i'}=[w_i,Q_{i'}, w(Q_{i'})]\] are ordinary augmentations. By Corollary   \ref{1.1.16},  we have  $\Phi(w_i,w)=\Phi(w_i,w_{i'})$ and 
	as $Q_i\notin\Phi(w_{i},w),$ so 
		\begin{align}\label{1.25}
			w_i(Q_{i})=w_{i'}(Q_i),~ \forall ~ i'>i~ \text{in}~\vartheta_{j},
		\end{align}
		which in view of  Theorem \ref{1.1.19}, implies that $Q_{i'}\not\sim_{w_i} Q_i.$
		Hence $\mathcal{W}_j:=(w_i)_{i\in\vartheta_j}$ is a continuous family of augmentations such that for each $i\in\vartheta_{j},$ $w_i$ is an ordinary augmentation of $w_j$ with respect to  $w(Q_i)$ and from (\ref{1.25}),  $Q_i$ is $\mathcal{W}_j$-stable  with stability degree $m_j.$ As $\vartheta_{j}\neq\emptyset,$ so  by Definition \ref{1.1.22}, $Q_{j+1}$ is a limit ABKP for $w.$  Since  $\delta(Q_i)<\delta(Q_{i'}),$ for every $i<i'\in\vartheta_{j},$ and    $w_{i'}(Q_{j+1})<w(Q_{j+1})$ for every $i'\in\vartheta_j,$ (because  $i'<j+1\in\Delta$), so  in view of Proposition \ref{2.1.6} (iii), we have   that $$w_i(Q_{j+1})<w_{i'}(Q_{j+1})~ \text{for every $i<i'\in \vartheta_{j}$}.$$  We now show that $ Q_{j+1}$ has  minimal degree with this property. For any polynomial $f\in K[X]$ with $\deg f<\deg Q_{j+1},$  there exists some ABKP $Q_{i_0}\in\Lambda$ such that $\deg Q_{i_0}\leq\deg f$ and 
		$	w_{i_0}(f)=w(f).$
	 Clearly,  $i_0<j+1.$
		 Therefore, for any $i,~i'\in\vartheta_{j}$ with $i_0<i<i',$ Proposition \ref{2.1.6} (iii) shows that \[w_{i_0}(f)\leq w_i(f)\leq w_{i'}(f)\leq w(f)=w_{i_0}(f),\] so that $f$ is $\mathcal{W}_j$-stable. Therefore, $Q_{j+1}$ is a an MLV limit key polynomial for $\mathcal{W}_j.$ Now, since $\deg(\mathcal{W}_j)=\deg Q_j<\deg Q_{j+1},$ we see that $\mathcal{W}_j$ is an essential continuous family of augmentations of  $w_j.$ 
		 
		 Let $\gamma_{j+1}$ denote the value $w(Q_{j+1}).$ Clearly, $\gamma_{j+1}>w_i(Q_{j+1})$ for all $i\in\vartheta_{j}.$ Let \[\rho_{j+1}=[\mathcal{W}_j;Q_{j+1},\gamma_{j+1}]\] be the limit augmentation determined by these data. Since $Q_{j+1}$ is a key polynomial of minimal degree for $\rho_{j+1},$ Theorem \ref{2.1.2} shows that $\rho_{j+1}$ is the truncation \[\rho_{j+1}=w_{Q_{j+1}}=w_{j+1}.\] Thus, the step $w_j\longrightarrow w_{j+1}$ is a limit augmentation. The condition of a an MLV chain is fulfilled because $Q_j\notin\Phi(w_j,w)=\Phi(w_j,w_{j+1}),$ the last equality follows from Corollary \ref{1.1.16}.

 Clearly, $w_{N}=w,$  for if there exists some polynomial $f\in K[X]$ such that $w_{N}(f)<w(f),$ then as $\Lambda$ is a complete set, so $w_{i}(f)=w(f)$ for some $0\leq i<N.$ But this will imply that $w(f)=w_{i}(f)\leq w_{N}(f)<w(f).$
		
		Thus from the above arguments it follows that $$w_0\xrightarrow{Q_1,\gamma_1} w_1\xrightarrow{Q_2,\gamma_2}\cdots \longrightarrow w_{N-1}\xrightarrow{Q_N,\gamma_N} w_N=w,$$  where $w_i=w_{Q_i},$ $\gamma_i=w(Q_i)$ for  $0\leq i\leq N,$ is a an MLV chain  whose last valuation is $w_N=w,$ such that:
		\begin{itemize}
			\item if $\vartheta_{j}= \emptyset,$ then $w_j\longrightarrow w_{j+1}$ is an ordinary augmentation.
			\item if $\vartheta_{j}\neq \emptyset,$ then $w_j\longrightarrow w_{j+1}$ is a limit augmentation of an essential continuous family $\mathcal{W}_j=(w_i)_{i\in\vartheta_j}$ of augmentations of $w_j.$
		\end{itemize}
		Finally the chain is complete because $w_0=w_{Q_0},$ where $Q_0=X,$   defined by the pair $(0,w_0(X))$ and is a depth zero valuation.
		\end{proof}

	\begin{proof}[Proof of Theorem \ref{1.1.12}]
		Let $$	w_0\xrightarrow{\phi_1,\gamma_1} w_1\xrightarrow{\phi_2,\gamma_2}\cdots \longrightarrow w_{N-1}\xrightarrow{\phi_N,\gamma_N} w_N=w,$$ be a complete finite MLV chain of $w.$ If $N=0,$  then $w_0=w$ is a depth zero valuation and the result holds trivially.
		
		 Assume now that $N\geq 1.$
			Then each $\phi_j$ is a key polynomial for $w_j$ of minimal degree, i.e., $\deg \phi_j=\deg(w_j),$ and  $w_j(\phi_{j})=w(\phi_j).$   Therefore $\phi_j\notin\Phi(w_j,w)$ which in view of Theorem \ref{2.1.2}, implies that $\phi_j$ is an ABKP for $w$ and
			\begin{align}\label{1.26}
			 w_j=w_{\phi_j}~\text{for  $0\leq j\leq N.$ } 
		\end{align}
		Suppose first that $w_j\longrightarrow w_{j+1}$ is an ordinary augmentation. Then by definition of MLV chain of $w,$ we have  $\phi_{j+1}\in \Phi(w_j,w_{j+1})=\Phi(w_j,w),$ the last equality  by Corollary \ref{1.1.16}.
		 Now, keeping in mind (\ref{1.26}), this implies that
		 $\phi_{j+1}\in \Phi(w_{\phi_j},w)$ and hence from Lemma \ref{1.2.6}, it follows that $$\delta(\phi_j)<\delta(\phi_{j+1}).$$
		
		Assume now that $w_j\longrightarrow w_{j+1}$ is a limit augmentation. Then $\phi_{j+1}$  is a an MLV limit key polynomial for an essential continuous family (say) $\mathcal{W}_j$ of augmentations of $w_j.$ 
		Let $\mathcal{W}_j=(\rho_i)_{i\in \mathbf{A}_j},$  where $\mathbf{A}_j$ is some totally ordered set without a last element and for each $i\in\mathbf{A}_j,$  $\rho_i=[w_j, \chi_i, \gamma_{i}]$ is an ordinary augmentation of $w_j$ with stability degree, (say) $m_j=\deg \phi_j=\deg \chi_i.$ Also, for all $i<i'\in\mathbf{A}_j,$ $\chi_{i'}$ is a key polynomial for $\rho_i$ such that $$\chi_{i'}\not\thicksim_{\rho_i}\chi_i,~ \text{i.e.,}~ \chi_{i'}\nmid_{\rho_i}\chi_i~\text{and}~ \rho_{i'}=[\rho_i; \chi_{i'}, \gamma_{i'}].$$  Since $\rho_i<w$ and $\chi_{i'}\nmid_{\rho_i} \chi_i,$ so by Theorem \ref{1.1.19}, $\rho_i(\chi_i)=w(\chi_i),$ which together with Corollary \ref{1.1.16}, gives 
		\begin{align}\label{1.24}
			\chi_i\notin\Phi(\rho_i,w)=\Phi(\rho_i,\rho_{i'})=[\chi_{i'}]_{\rho_i}.
		\end{align}
	Now by Remark \ref{1.1.18} (i), for each $i\in \mathbf{A}_j,$ $\chi_i$ is a key polynomial for $\rho_i$ of minimal degree, i.e.,  $\deg \chi_i=\deg(\rho_i),$  therefore keeping in mind that $\rho_i<w,$ equation (\ref{1.24}) in view of Theorem \ref{2.1.2},  implies that each $Q_i:=\chi_i$ is an ABKP for $w$ and
	\begin{align}\label{1.1}
	 \rho_i=w_{Q_i},~\text{for all $i\in\mathbf{A}_j$}.
	\end{align}
	  Hence for each $i<i'\in\mathbf{A}_j,$ $Q_i$ and $Q_{i'}$ are ABKPs for $w$ such that $$w(Q_i)=w_{Q_i}(Q_i)<w_{Q_{i'}}(Q_{i'})=w(Q_{i'})~ \text{and}~  \deg Q_i=\deg Q_{i'},$$ which in view of Proposition \ref{2.1.6} (ii), implies that $w_{Q_i}(Q_{i'})<w(Q_{i'}).$ Therefore, by Lemma \ref{1.2.6}, we get that  $$Q_{i'}\in\Phi(w_{Q_i}, w)~\text{and}~ \delta(Q_i)<\delta(Q_{i'})~\text{for every $i<i'\in \mathbf{A}_{j}$}.$$ As $\deg \phi_j=\deg Q_i,$ for every $i\in\mathbf{A}_{j}$ and $w(\phi_j)<w(Q_i),$ so again by Proposition \ref{2.1.6} (ii), we have that   $Q_i\in\Phi(w_{\phi_j}, w)$  and then  $$\delta(\phi_j)<\delta(Q_i)~\forall ~i\in\mathbf{A}_j,$$ follows from Lemma \ref{1.2.6}.   Since $\mathcal{W}_j$ is essential, so $\deg \phi_{j}=\deg Q_i<\deg \phi_{j+1},$ for every $i\in\mathbf{A}_j,$ this together with the fact that $\phi_{j+1}$ is an ABKP for $w,$ implies that $$\delta(\phi_{j})<\delta(\phi_{j+1}), ~ \delta(Q_i)<\delta(\phi_{j+1}) ~\text{and}~ \phi_{j+1}\notin \Phi(w_{\phi_j},w).$$
	
	For every $0\leq j\leq N,$ let $\Delta_j=\{j\}\cup\mathbf{A}_j,$  and $\Delta=\bigcup_{j=0}^{N}\Delta_j.$ We now set $Q_j:=\phi_j$ for all $j\in I=\{0,1,\ldots,N\},$ and  show that  $\Lambda=\{Q_i\}_{i\in\Delta}$ is a complete set of ABKPs for $w.$ Clearly, as shown above for every $i<i'\in\Delta,$ we have $\delta(Q_i)<\delta(Q_{i'}).$ Therefore,  the set $\Lambda$ is well-ordered with respect to the ordering given by: $Q_i< Q_{i'}$ if and only if  $\delta(Q_i)<\delta(Q_{i'})$ for every $i<i'\in \Delta.$ 
	 Let $f\in K[X]$ be any polynomial.  Then on using the fact that  $w_N=w$ and  (\ref{1.26}),  we get  $w_{Q_N}=w;$ that is $w_{Q_N}(f)=w(f).$ If $ \deg Q_N\leq \deg f,$ then 
	 we are done.  
	 
	 So suppose that $\deg Q_j\leq\deg f<\deg Q_{j+1},$ for some $j\in I\setminus\{N\}.$ If $w_j\longrightarrow w_{j+1}$ is an ordinary augmentation, then by
	  Remark \ref{1.1.20} (iii) and (iv),
	   we have that $f\notin\Phi(w_{Q_j},w_{Q_{j+1}}),$ which together with Corollary \ref{1.1.16}    implies that $f\notin\Phi(w_{Q_j},w)$ and consequently,  $w_{Q_j}(f)=w(f).$ Now if  $w_j\longrightarrow w_{j+1}$ is a limit augmentation, then as $Q_{j+1}$ is a $\mathcal{W}_j$-unstable polynomial of minimal degree, so we must have that $\rho_i(f)=\rho_{i'}(f)$ for all $i>i'\in\mathbf{A}_j,$  for some index $i'\in\mathbf{A}_j,$ which in view of (\ref{1.1}) implies that  
	  $w_{Q_i}(f)=w_{Q_{i'}}(f),$ i.e., $f\notin\Phi(w_{Q_{i'}},w_{Q_{i}})=\Phi(w_{Q_{i'}},w)$  (see Corollary \ref{1.1.16}).      Hence $w_{Q_{i'}}(f)=w(f),$ for $i'\in\mathbf{A}_j.$ 
	  Thus  $\{Q_i\}_{i\in\Delta}$ is  a complete set of ABKPs for $w$ such that   
	 \begin{itemize}
	 	\item if $w_j\longrightarrow w_{j+1}$ is an ordinary augmentation, then $\vartheta_j=\emptyset.$ 
	 	\item if $w_j\longrightarrow w_{j+1}$ is a limit augmentation, 
	 	with respect to an essential continuous family $\mathcal{W}_j=(\rho_i)_{i\in\mathbf{A}_j}$ of augmentations of $w_j,$ then $\vartheta_{j}=\mathbf{A}_j,$  $Q_i=\chi_i$ for all $i\in\mathbf{A}_j,$ and $Q_{j+1}=\phi_{j+1}$ is a limit ABKP for $w.$
	 \end{itemize}
	\end{proof}

	\section*{Acknowledgement}
		We would like to thank the anonymous referee for a careful reading and providing useful suggestions which led to an improvement in the presentation of the paper.
	The	research of  first author is supported by CSIR (Grant No.\  09/045(1747)/2019-EMR-I).


\begin{thebibliography}{99}
	\bibitem{Ma}  M.\ Alberich-Carramin\~ ana, A.\ F.\ F.\  Boix, J.\ Fern\'andez, J.\ Gu\`ardia, E.\ Nart and J.\  Ro\'e,  Of limit key polynomials.  {\it Illinois J.\  Math.} {\bf 65} (2021), no.\ 1, 201–229.
			\bibitem{FV-K} F.\ -V.\  Kuhlmann,  Value groups, residue fields, and bad places of rational function fields. {\it Trans.\  Amer.\  Math.\  Soc.} {\bf 356} (2004), no.\  11, 4559-4600.
				\bibitem{M} S.\  MacLane,  A construction for absolute values in polynomial rings. {\it Trans.\  Amer.\  Math.\  Soc.} {\bf 40} (1936), no.\  3, 363-395.
				\bibitem{MMS}  W.\ Mahboub, A.\  Mansour and M.\ Spivakovsky,  On common extensions of valued fields. {\it J.\  Algebra}  {\bf 584} (2021), 1-18.
	\bibitem{EN}  E.\ Nart,  Key polynomials over valued fields. {\it Publ.\  Mat.}  {\bf 64} (2020), No.\ 1, 195–232.
		\bibitem{EN1} E.\ Nart, MacLane-Vaquié chains of valuations on a polynomial ring.  {\it  Pac.\	J.\ Math.} {\bf 311} (2021), 165-195.
		\bibitem {NS} J.\ Novacoski and M.\  Spivakovsky,  Key polynomials and pseudo-convergent sequences. {\it J.\  Algebra} {\bf 495} (2018), 199-219.
		\bibitem {JN} J.\  Novacoski,  Key polynomials and minimal pairs. {\it J.   Algebra} {\bf 523} (2019), 1-14.
		\bibitem{JN1} J.\ Novacoski,  On MacLane-Vaqui\'e key polynomials. {\it J.\   Pure Appl.\  Algebra} {\bf 225} (2021), no.\  8, 106644.
		\bibitem{V}  M.\ Vaqui\'e,   Extension d'une valuation. {\it Trans.\  Amer.\  Math.\  Soc.} {\bf 359} (2007), no.\  7, 3439-3481
	\end{thebibliography}
\end{document}